\documentclass[10pt]{amsart} 
\usepackage[foot]{amsaddr}
\usepackage{amsmath,amssymb,bm, color,float}

\usepackage{graphicx}

\usepackage{mathrsfs}
\usepackage{color}
\usepackage{float}
\usepackage{amsfonts,amssymb}
\usepackage{dsfont}
\usepackage{pifont}
\usepackage{wrapfig} 
\usepackage{hyperref}
\usepackage{multirow}
\numberwithin{equation}{section}
\def\3bar{{|\hspace{-.02in}|\hspace{-.02in}|}}

\def\W{{\mathcal{W}}}

\def\bn{{\mathbf{n}}}

\def\bbeta{{\boldsymbol{\beta}}}

\newtheorem{algorithm}{WG-LS Algorithm}[section]

\numberwithin{equation}{section}
 \newtheorem{theorem}{Theorem}[section]
 \newtheorem{lemma}[theorem]{Lemma}

   \def\t#1{\operatorname{#1}}
\def\a#1{\begin{align*}#1\end{align*}} \def\an#1{\begin{align}#1\end{align}} 
 
\def\p#1{\begin{pmatrix}#1\end{pmatrix}}

\title[Weak Galerkin Finite Element]
{A weak Galerkin least squares finite element method for linear convection equations in non-divergence form}

  \author {Chunmei Wang} 
  \address{Department of Mathematics, University of Florida, Gainesville, FL 32611, USA. }
  \email{chunmei.wang@ufl.edu} 
  
\author {Shangyou Zhang}
\address{Department of Mathematical Sciences,  University of Delaware, Newark, DE 19716, USA}   \email{szhang@udel.edu}

\begin{document}

\begin{abstract}
This article develops a weak Galerkin least-squares (WG--LS) finite element method for first-order linear convection equations in non-divergence form. The method is formulated using discontinuous finite element functions and does not require any coercivity assumption on the convection vector or reaction coefficient. The resulting discrete problem leads to a symmetric and positive definite linear system and is applicable to general polygonal and polyhedral meshes. Under minimal regularity assumptions on the coefficients, optimal-order error estimates are established for the WG--LS approximation in a suitable energy norm. Numerical experiments are presented to confirm the theoretical convergence results and to demonstrate the accuracy and efficiency of the proposed method.
\end{abstract}

 
 
\maketitle

\section{Introduction}

This paper is concerned with the development of a new numerical method for first-order linear convection equations in non-divergence form using discontinuous finite element functions. For simplicity, we consider the model problem of finding an unknown function $u$ such that
\begin{equation}\label{model}
\begin{split}
  \bbeta \cdot \nabla u + c(x)u &= f \quad \text{in } \Omega, \\
  u &= g \quad \text{on } \Gamma_-,
\end{split}
\end{equation}
where $\Omega$ is an open, bounded, and connected domain in $\mathbb{R}^d$ ($d=2,3$) with Lipschitz continuous boundary $\Gamma=\partial\Omega$. The inflow boundary $\Gamma_- \subset \Gamma$ is defined by the condition $\bbeta \cdot \bn < 0$, where $\bn$ denotes the unit outward normal vector on $\Gamma$. We assume that the convection vector $\bbeta \in[L^\infty(\Omega)]^d$, the reaction coefficient $c\in L^\infty(\Omega)$, the source term $f\in L^2(\Omega)$, and the inflow boundary data $g\in L^2(\Gamma_-)$.

First-order linear partial differential equations (PDEs) of hyperbolic type, also referred to as transport or linear convection equations, arise in many areas of science and engineering, including fluid dynamics and neutron transport. Over the past several decades, substantial research effort has been devoted to the development of accurate and efficient numerical methods for hyperbolic problems. Due to localized phenomena such as propagating discontinuities and sharp transition layers, the numerical approximation of hyperbolic equations remains challenging. In particular, since linear hyperbolic PDEs generally admit discontinuous solutions for nonsmooth boundary data, it is difficult to design numerical methods that simultaneously achieve high-order accuracy in smooth regions and sharp resolution of discontinuities while avoiding spurious oscillations near nonsmooth features \cite{bd1999}. Moreover, linear hyperbolic equations serve as prototype models for more general hyperbolic systems, including nonlinear conservation laws \cite{l1992} and transport equations in phase space \cite{lm1993}. Effective numerical methods for linear convection equations can therefore serve as essential building blocks for the numerical solution of complex hyperbolic PDEs \cite{l1992}. 

A wide range of numerical methods has been developed for linear transport equations, including streamline-upwind Petrov--Galerkin methods \cite{eehj1996}, residual distribution methods \cite{jnp1984,eehj1996,a2001}, least-squares finite element methods \cite{cj1988,bc2001,hjs2002,smmo2004,b1999}, stabilized finite element methods \cite{burman2014}, and various discontinuous Galerkin methods \cite{rh,lr,c1999,cs,hss,j2004,eg,bh,b,be,bs,bs2007}. Most existing studies of the linear transport equation \eqref{model} rely on a coercivity condition on the convection vector $\bbeta$ and the reaction coefficient $c$ of the form
\[
c + \tfrac12 \nabla \cdot \bbeta \ge \alpha_0,
\]
for some positive constant $\alpha_0$, or a similar assumption. Such conditions are often restrictive in practical applications and may exclude important physical models, including compressible flows and exothermic reaction processes \cite{burman2014}. One notable exception is the stabilized finite element method developed by Burman \cite{burman2014}, which assumes $\bbeta\in[W^{2,\infty}(\Omega)]^d$ and $c\in W^{1,\infty}(\Omega)$.

The weak Galerkin (WG) finite element method, originally introduced in \cite{ellip_JCAM2013} and further developed in \cite{wg1,wg2,wg3,wg4,wg5,wg6,wg7,wg8,wg9,wg10,wg11,wg12,wg13,wg14,wg15,wg16,wg17,wg18,wg19,wg20,wg21,itera,wz2023,wy3655}, is based on the use of weak derivatives and weak continuity, leading to numerical schemes that are stable and highly flexible with respect to general polygonal and polyhedral meshes. An important extension of the WG framework is the primal-dual weak Galerkin (PDWG) method \cite{pdwg1,pdwg2,pdwg3,pdwg4,pdwg5,pdwg6,pdwg7,pdwg8,pdwg9,pdwg10,pdwg11,pdwg12,pdwg13,pdwg14,pdwg15}. In the PDWG approach, numerical approximations are formulated as constrained minimization problems in which the governing equations are enforced through weak constraints. The resulting Euler--Lagrange systems involve both primal and dual variables and possess favorable symmetry and stability properties. PDWG methods have also been developed for linear transport equations; see, for example, \cite{wwhyperbolic,pdwg1,pdwg6,pdwg12}.

The objective of this paper is to develop a new numerical method for the linear convection problem \eqref{model}, where the convection vector $\bbeta$ and the reaction coefficient $c$ are assumed to be only piecewise smooth, without imposing any additional coercivity condition such as the one above. To this end, we propose a \emph{weak Galerkin least-squares (WG--LS) finite element method} that is easy to implement and leads to a discretized linear system that is symmetric and positive definite, distinguishing it from many existing weak Galerkin formulations; see, for example, \cite{wwhyperbolic,pdwg1,pdwg6,pdwg12}. The proposed WG--LS method is rigorously analyzed with respect to stability and convergence. Moreover, the associated theoretical analysis is established under minimal assumptions on the linear convection problem, requiring only the uniqueness of the continuous solution and piecewise smoothness of the coefficients.

The remainder of the paper is organized as follows. In Section~2, we briefly review the weak gradient and its discrete counterpart. In Section~3, we present the WG--LS algorithm for the first-order linear convection problem in non-divergence form. In Section~4, we establish the existence and uniqueness of the numerical solution. Section~5 is devoted to the derivation of the error equations for the WG--LS finite element approximation. In Section~6, we prove optimal-order error estimates in a discrete Sobolev norm. Finally, in Section~7, numerical results are reported to demonstrate the stability, accuracy, and efficiency of the proposed WG--LS method.

Throughout the paper, we adopt standard notation for Sobolev spaces and norms. For any open bounded domain $D\subset\mathbb{R}^d$ with Lipschitz continuous boundary, $\|\cdot\|_{s,D}$ and $|\cdot|_{s,D}$ denote the norm and seminorm of the Sobolev space $H^s(D)$ for $s\ge0$, respectively, and $(\cdot,\cdot)_{s,D}$ denotes the corresponding inner product. In particular, $H^0(D)=L^2(D)$, with norm $\|\cdot\|_D$ and inner product $(\cdot,\cdot)_D$. When $D=\Omega$, or when the domain is clear from the context, the subscript $D$ is omitted.

 \section{Discrete Weak Gradient}

Let $T$ be a polygonal (in two dimensions) or polyhedral (in three dimensions) domain with boundary $\partial T$. A \emph{weak function} on $T$ is defined as an ordered pair
$v=\{v_0,v_b\}$, where $v_0\in L^2(T)$ represents the value of $v$ in the interior of $T$, and
$v_b\in L^2(\partial T)$ represents the value of $v$ on the boundary $\partial T$. In general, $v_b$ is not required to be the trace of $v_0$ on $\partial T$, although this choice is admissible.

We denote by $\W(T)$ the space of all weak functions on $T$, defined by
\begin{equation}\label{eq:weakspace}
\W(T)=\bigl\{v=\{v_0,v_b\}: v_0\in L^2(T),\ v_b\in L^2(\partial T)\bigr\}.
\end{equation}

The \emph{weak gradient} of a function $v\in \W(T)$, denoted by $\nabla_w v$, is defined as a linear functional acting on $[H^1(T)]^d$ such that
\begin{equation}\label{eq:weakgrad}
(\nabla_w v,\boldsymbol{\psi})_T
:=-(v_0,\nabla\cdot\boldsymbol{\psi})_T
+\langle v_b,\boldsymbol{\psi}\cdot\mathbf{n}\rangle_{\partial T},
\qquad \forall \boldsymbol{\psi}\in [H^1(T)]^d,
\end{equation}
where $\mathbf{n}$ denotes the outward unit normal vector on $\partial T$.

Let $P_r(T)$ be the space of polynomials of degree at most $r$ on $T$. The \emph{discrete weak gradient} of $v\in\W(T)$, denoted by $\nabla_{w,r,T}v$, is defined as the unique vector-valued polynomial in $[P_r(T)]^d$ satisfying
\begin{equation}\label{disgradient}
(\nabla_{w,r,T} v,\boldsymbol{\psi})_T
=-(v_0,\nabla\cdot\boldsymbol{\psi})_T
+\langle v_b,\boldsymbol{\psi}\cdot\mathbf{n}\rangle_{\partial T},
\qquad \forall \boldsymbol{\psi}\in [P_r(T)]^d.
\end{equation}

If $v_0\in H^1(T)$, then by integration by parts, \eqref{disgradient} can be equivalently written as
\begin{equation}\label{disgradient_star}
(\nabla_{w,r,T} v,\boldsymbol{\psi})_T
=(\nabla v_0,\boldsymbol{\psi})_T
-\langle v_0-v_b,\boldsymbol{\psi}\cdot\mathbf{n}\rangle_{\partial T},
\qquad \forall \boldsymbol{\psi}\in [P_r(T)]^d.
\end{equation}
 \section{WG--LS Algorithm}\label{Section:WGFEM}

Let $\mathcal{T}_h$ be a shape-regular partition of the domain $\Omega$ into polygonal (2D) or polyhedral (3D) elements in the sense of \cite{wy3655}. Let $\mathcal{E}_h$ denote the set of all edges/faces in $\mathcal{T}_h$, and let
$\mathcal{E}_h^0=\mathcal{E}_h\setminus\partial\Omega$ be the set of all interior edges/faces. For each $T\in\mathcal{T}_h$, denote by $h_T$ the diameter of $T$, and define the mesh size $h=\max_{T\in\mathcal{T}_h}h_T$.

For a given integer $k\ge1$, define the local weak finite element space
\[
W_k(T)=\bigl\{\{\sigma_0,\sigma_b\}:\sigma_0\in P_k(T),\ \sigma_b\in P_k(e),\ e\subset\partial T\bigr\}.
\]
By patching these local spaces together with single-valued boundary components on interior edges/faces, we obtain the global weak finite element space $W_h$. Let $W_h^0$ be the subspace of $W_h$ consisting of functions with vanishing boundary values on the inflow boundary $\Gamma_-$, i.e.,
\[
W_h^0=\bigl\{\{\sigma_0,\sigma_b\}\in W_h:\ \sigma_b|_e=0,\ e\subset\Gamma_-\bigr\}.
\]

For simplicity of notation, for any $\sigma\in W_h$, we denote by $\nabla_w\sigma$ the discrete weak gradient defined elementwise by
\[
(\nabla_w\sigma)|_T=\nabla_{w,k,T}(\sigma|_T),
\qquad \forall T\in\mathcal{T}_h.
\]

We define the following bilinear forms:
\[
s(u,v)=\sum_{T\in\mathcal{T}_h}\int_{\partial T}h_T^{-1}(u_0-u_b)(v_0-v_b)\,ds,
\]
\[
a(u,v)=\sum_{T\in\mathcal{T}_h}
\bigl(\bbeta\cdot\nabla_w u + cu_0,\ \bbeta\cdot\nabla_w v + cv_0\bigr)_T,
\]
for all $u,v\in W_h$.

We are now ready to state the WG--LS finite element scheme for the first-order linear convection problem \eqref{model}.

\begin{algorithm}\label{PDWG1}
Find $u_h\in W_h$ with $u_b=Q_b g$ on $\Gamma_-$ such that
\begin{equation}\label{al-general}
a(u_h,v_h)+s(u_h,v_h)
=(f,\bbeta\cdot\nabla_w v_h + cv_0),
\qquad \forall v_h\in W_h^0,
\end{equation}
where $Q_b$ denotes the local $L^2$ projection onto $P_k(e)$ on each edge/face $e$.
\end{algorithm}

\section{Solution Existence and Uniqueness}\label{Section:EU}

On each element $T\in\mathcal{T}_h$, let $Q_0$ denote the $L^2$ projection onto $P_k(T)$, where $k\ge1$ is a fixed integer. On each edge/face $e\subset\partial T$, let $Q_b$ be the $L^2$ projection onto $P_k(e)$. For any $w\in H^1(\Omega)$, define the local $L^2$ projection $Q_h w\in W_h$ by
\[
(Q_h w)|_T := \{Q_0(w|_T),\, Q_b(w|_{\partial T})\}, 
\qquad \forall T\in\mathcal{T}_h.
\] 

\begin{lemma}\label{Lemma5.1}
The $L^2$ projection operators $Q_h$ and $Q_0$ satisfy the following commutative property:
\begin{equation}\label{eq:commute}
\nabla_w(Q_h w)=Q_0(\nabla w), \qquad \forall w\in H^1(T).
\end{equation}
\end{lemma}

\begin{proof}
For any $\boldsymbol{\psi}\in [P_k(T)]^d$, using \eqref{disgradient} we obtain
\[
\begin{aligned}
(\nabla_w Q_h w,\boldsymbol{\psi})_T
&=-(Q_0 w,\nabla\cdot\boldsymbol{\psi})_T
 +\langle Q_b w,\boldsymbol{\psi}\cdot\mathbf{n}\rangle_{\partial T} \\
&=-(w,\nabla\cdot\boldsymbol{\psi})_T
 +\langle w,\boldsymbol{\psi}\cdot\mathbf{n}\rangle_{\partial T} \\
&=(\nabla w,\boldsymbol{\psi})_T
=(Q_0\nabla w,\boldsymbol{\psi})_T,
\end{aligned}
\]
which proves \eqref{eq:commute}.
\end{proof}

\begin{lemma}\label{thmunique1}
Assume that the first-order linear convection problem \eqref{model} admits a unique solution. Then the WG--LS scheme \eqref{al-general} has a unique solution.
\end{lemma}

\begin{proof}
It suffices to show that the homogeneous problem associated with \eqref{al-general} admits only the trivial solution. To this end, assume $f=0$ and $g=0$.  Thus,  $u_h\in W_h^0$ satisfies
\[
a(u_h,v_h)+s(u_h,v_h)=0,
\qquad \forall v_h\in W_h^0.
\]
Choosing $v_h=u_h$ yields
\[
a(u_h,u_h)+s(u_h,u_h)=0,
\]
which implies
\[
u_0=u_b \quad \text{on } \partial T, 
\qquad 
\bbeta\cdot\nabla_w u_h + cu_0=0 \quad \text{on } T,
\quad \forall T\in\mathcal{T}_h.
\]

Using \eqref{disgradient_star} and the identity $u_0=u_b$ on $\partial T$, we have
\[
(\nabla_w u_h,\boldsymbol{\psi})_T=(\nabla u_0,\boldsymbol{\psi})_T,
\qquad \forall \boldsymbol{\psi}\in[P_k(T)]^d,
\]
which implies $\nabla_w u_h=\nabla u_0$ on each element $T$. Consequently,
\[
\bbeta\cdot\nabla u_0 + cu_0=0 \quad \text{in } T,\ \forall T\in\mathcal{T}_h.
\]

Since $u_0=u_b$ on $\partial T$ for all $T$, it follows that $u_0\in C^0(\Omega)$. Moreover, the boundary condition $u_b=0$ on the inflow boundary $\Gamma_-$ implies $u_0=0$ on $\Gamma_-$. By the uniqueness of the continuous problem \eqref{model}, we conclude that $u_0\equiv0$ in $\Omega$, and hence $u_b\equiv0$. Therefore, $u_h\equiv0$, which completes the proof.
\end{proof}

We define a semi-norm on $W_h$ by
\[
\3bar v \3bar^2 := a(v,v).
\]
By an argument similar to that of Lemma~\ref{thmunique1}, one can show that $\3bar \cdot \3bar$ defines a norm on $W_h^0$.

\section{Error Equations}

Let $u$ be the exact solution of \eqref{model}, and let $u_h\in W_h$ be the WG--LS approximation defined by \eqref{al-general}. We define the error function
\begin{equation}\label{error}
e_h := u_h - Q_h u.
\end{equation}

For simplicity of presentation, we assume that $\bbeta$ and $c$ are piecewise constants in the following analysis. The extension to piecewise smooth coefficients follows similarly.

\begin{lemma}\label{errorequa}
The error function $e_h$ satisfies the following error equation:
\begin{equation}\label{erroreqn}
a(e_h,v_h)+s(e_h,v_h) = -s(Q_h u,v_h),
\qquad \forall v_h\in W_h^0.
\end{equation}
\end{lemma}

\begin{proof}
Testing the continuous problem \eqref{model} with $\bbeta\cdot\nabla_w v_h + cv_0$ yields
\[
(\bbeta\cdot\nabla u + cu,\ \bbeta\cdot\nabla_w v_h + cv_0)
=(f,\ \bbeta\cdot\nabla_w v_h + cv_0).
\]
Since $\bbeta\cdot\nabla_w v_h + cv_0 \in P_k(T)$ on each element $T$, applying the commutative property \eqref{eq:commute} gives
\[
(\bbeta\cdot\nabla_w Q_h u + cQ_0 u,\ \bbeta\cdot\nabla_w v_h + cv_0)
=(f,\ \bbeta\cdot\nabla_w v_h + cv_0).
\]
Subtracting the discrete formulation \eqref{al-general} from the above identity yields
\[
a(e_h,v_h)+s(u_h,v_h)=0,
\]
which implies
\[
a(e_h,v_h)+s(e_h,v_h)=-s(Q_h u,v_h),
\]
completing the proof.
\end{proof}

\section{Error Estimates}

Recall that $\mathcal{T}_h$ is a shape-regular finite element partition of the domain $\Omega$. Consequently, for any element $T\in\mathcal{T}_h$ and any function $\phi\in H^1(T)$, the following trace inequality holds \cite{wy3655}:
\begin{equation}\label{tracein}
\|\phi\|_{\partial T}^2 \le C\big(h_T^{-1}\|\phi\|_T^2 + h_T\|\nabla \phi\|_T^2\big).
\end{equation}
Moreover, if $\phi$ is a polynomial on $T$, then the sharper estimate
\begin{equation}\label{trace}
\|\phi\|_{\partial T}^2 \le C h_T^{-1}\|\phi\|_T^2
\end{equation}
holds true \cite{wy3655}.

\begin{lemma}\cite{wy3655}
Let $\mathcal{T}_h$ be a shape-regular finite element partition of $\Omega$. For any $0\le s\le1$ and $0\le n\le k$, there exists a constant $C>0$ such that
\begin{equation}\label{error2}
\sum_{T\in\mathcal{T}_h} h_T^{2s}\|u-Q_0 u\|_{s,T}^2
\le C h^{2n+2}\|u\|_{n+1}^2.
\end{equation}
\end{lemma}

\begin{lemma}\label{lem:H1error}
Assume that the exact solution $u$ of the linear convection problem \eqref{model} satisfies $u\in H^{k+1}(\Omega)$. Then there exists a constant $C>0$, independent of $h$, such that
\begin{equation}\label{erroresti1}
\3bar e_h \3bar \le C h^{k}\|u\|_{k+1}.
\end{equation}
\end{lemma}

\begin{proof}
Choosing $v_h=e_h$ in the error equation \eqref{erroreqn} yields
\[
\3bar e_h \3bar^2 = -s(Q_h u,e_h).
\]
Applying the Cauchy--Schwarz inequality and the definition of the stabilization term gives
\[
\3bar e_h \3bar^2
\le
\Bigg(\sum_{T\in\mathcal{T}_h} h_T^{-1}\|Q_0u-Q_bu\|_{\partial T}^2\Bigg)^{1/2}
\3bar e_h \3bar.
\]

Using the trace inequality \eqref{tracein} together with the approximation estimate \eqref{error2} (with $n=k$ and $s=0,1$), we obtain
\[
\sum_{T\in\mathcal{T}_h} h_T^{-1}\|Q_0u-Q_bu\|_{\partial T}^2
\le
C\sum_{T\in\mathcal{T}_h}
\big(h_T^{-2}\|u-Q_0u\|_T^2 +  \|u-Q_0u\|_{1,T}^2\big)
\le
C h^{2k}\|u\|_{k+1}^2.
\]
Combining the above estimates and canceling $\3bar e_h \3bar$ from both sides yields
\[
\3bar e_h \3bar \le C h^{k}\|u\|_{k+1},
\]
which completes the proof.
\end{proof}

\section{Numerical experiments}
In the first test,  we solve the convection equation \eqref{model} 
   on a square domain $\Omega= (-1,1)\times (-1,1)$, where  
\an{\label{lambda} \bbeta =\p{1\\1}, \quad c(x,y) =\lambda (x-\frac 12)(y-\frac 12).  }
By choosing $f$ and $g$ in \eqref{model}, the exact solution is 
\an{\label{u-1} u=\sin x \sin y. } 
            
\begin{figure}[H]
 \begin{center}\setlength\unitlength{1.0pt}
\begin{picture}(360,120)(0,0)
  \put(15,108){$G_1$:} \put(125,108){$G_2$:} \put(235,108){$G_3$:} 
  \put(0,-420){\includegraphics[width=380pt]{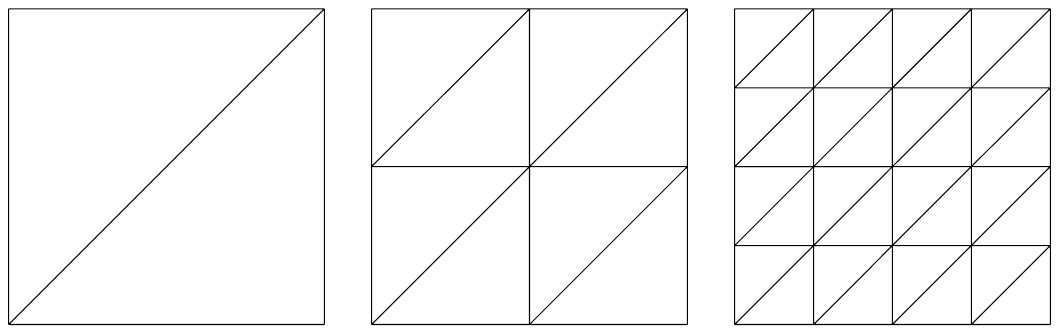}}  
 \end{picture}\end{center}
\caption{The triangular grids for the computation in Tables \ref{t1}--\ref{t4}. }\label{f21}
\end{figure}

The solution in \eqref{u-1} is approximated by the weak Galerkin finite element
   $P_k$-$P_k$/$P_{k+1}$ (for $\{u_0, u_b\}$/$\nabla_w$), $k= 1,2,3,4$, on triangular  
    grids shown in Figure \ref{f21}.
The errors and the computed orders of convergence are listed in Tables \ref{t1}--\ref{t4}.
The optimal order of convergence is achieved in every case.
   
  \begin{table}[H]
  \caption{By the $P_1$-$P_1$/$P_2$ element for \eqref{u-1}  on Figure \ref{f21} grids.} \label{t1}
\begin{center}  
   \begin{tabular}{c|rr|rr|rr}  
 \hline 
$G_i$ &  $ \|Q_h  u -   u_h \| $ & $O(h^r)$ &  $ \| \nabla_w( Q_h u- u_h )\| $ & $O(h^r)$ 
   &  $ \3bar  Q_h u- u_h  \3bar $ & $O(h^r)$  \\ \hline 
   &\multicolumn{6}{c}{$\lambda=1$ in \eqref{lambda} }\\  \hline 
 5&     0.649E-4 &  2.0&     0.126E-1 &  1.0&     0.168E-2 &  1.0 \\
 6&     0.162E-4 &  2.0&     0.630E-2 &  1.0&     0.842E-3 &  1.0 \\
 7&     0.406E-5 &  2.0&     0.315E-2 &  1.0&     0.421E-3 &  1.0 \\ 
 \hline  &\multicolumn{6}{c}{$\lambda=100$ in \eqref{lambda} }\\
 \hline 
 5&     0.442E-3 &  1.2&     0.138E-1 &  1.0&     0.174E-2 &  1.1 \\
 6&     0.153E-3 &  1.5&     0.666E-2 &  1.1&     0.852E-3 &  1.0 \\
 7&     0.443E-4 &  1.8&     0.322E-2 &  1.0&     0.423E-3 &  1.0 \\
  \hline  
\end{tabular} \end{center}  \end{table}
 
  \begin{table}[H]
  \caption{By the $P_2$-$P_2$/$P_3$ element for \eqref{u-1}  on Figure \ref{f21} grids.} \label{t2}
\begin{center}  
   \begin{tabular}{c|rr|rr|rr}  
 \hline 
$G_i$ &  $ \|Q_h  u -   u_h \| $ & $O(h^r)$ &  $ \| \nabla_w( Q_h u- u_h )\| $ & $O(h^r)$ 
   &  $ \3bar  Q_h u- u_h  \3bar $ & $O(h^r)$  \\ \hline 
   &\multicolumn{6}{c}{$\lambda=1$ in \eqref{lambda} }\\  \hline 
 3&     0.227E-4 &  2.7&     0.924E-3 &  1.8&     0.340E-3 &  2.0 \\
 4&     0.374E-5 &  2.6&     0.309E-3 &  1.6&     0.844E-4 &  2.0 \\
 5&     0.718E-6 &  2.4&     0.122E-3 &  1.3&     0.210E-4 &  2.0 \\
 \hline  &\multicolumn{6}{c}{$\lambda=100$ in \eqref{lambda} }\\
 \hline 
 3&     0.167E-3 &  1.8&     0.319E-2 &  1.3&     0.354E-3 &  3.5 \\
 4&     0.158E-4 &  3.4&     0.726E-3 &  2.1&     0.811E-4 &  2.1 \\
 5&     0.224E-5 &  2.8&     0.360E-3 &  1.0&     0.207E-4 &  2.0 \\
  \hline  
\end{tabular} \end{center}  \end{table}
 
  \begin{table}[H]
  \caption{By the $P_3$-$P_3$/$P_4$ element for \eqref{u-1}  on Figure \ref{f21} grids.} \label{t3}
\begin{center}  
   \begin{tabular}{c|rr|rr|rr}  
 \hline 
$G_i$ &  $ \|Q_h  u -   u_h \| $ & $O(h^r)$ &  $ \| \nabla_w( Q_h u- u_h )\| $ & $O(h^r)$ 
   &  $ \3bar  Q_h u- u_h  \3bar $ & $O(h^r)$  \\ \hline 
   &\multicolumn{6}{c}{$\lambda=1$ in \eqref{lambda} }\\  \hline 
 3&     0.366E-6 &  4.0&     0.403E-4 &  3.0&     0.696E-5 &  3.0 \\
 4&     0.227E-7 &  4.0&     0.503E-5 &  3.0&     0.866E-6 &  3.0 \\
 5&     0.141E-8 &  4.0&     0.629E-6 &  3.0&     0.108E-6 &  3.0 \\
 \hline  &\multicolumn{6}{c}{$\lambda=100$ in \eqref{lambda} }\\
 \hline 
 3&     0.143E-5 &  4.4&     0.537E-4 &  3.4&     0.100E-4 &  4.2 \\
 4&     0.375E-7 &  5.3&     0.529E-5 &  3.3&     0.947E-6 &  3.4 \\
 5&     0.254E-8 &  3.9&     0.679E-6 &  3.0&     0.111E-6 &  3.1 \\
  \hline  
\end{tabular} \end{center}  \end{table}
  
  \begin{table}[H]
  \caption{By the $P_4$-$P_4$/$P_5$ element for \eqref{u-1}  on Figure \ref{f21} grids.} \label{t4}
\begin{center}  
   \begin{tabular}{c|rr|rr|rr}  
 \hline 
$G_i$ &  $ \|Q_h  u -   u_h \| $ & $O(h^r)$ &  $ \| \nabla_w( Q_h u- u_h )\| $ & $O(h^r)$ 
   &  $ \3bar  Q_h u- u_h  \3bar $ & $O(h^r)$  \\ \hline 
   &\multicolumn{6}{c}{$\lambda=1$ in \eqref{lambda} }\\  \hline 
 2&     0.200E-6 &  5.2&     0.579E-5 &  4.2&     0.355E-5 &  4.2 \\
 3&     0.636E-8 &  5.0&     0.375E-6 &  3.9&     0.216E-6 &  4.0 \\
 4&     0.211E-9 &  4.9&     0.274E-7 &  3.8&     0.134E-7 &  4.0 \\
 \hline  &\multicolumn{6}{c}{$\lambda=100$ in \eqref{lambda} }\\
 \hline 
 2&     0.813E-6 &  6.4&     0.209E-4 &  5.2&     0.929E-5 &  6.5 \\
 3&     0.210E-7 &  5.3&     0.121E-5 &  4.1&     0.201E-6 &  5.5 \\
 4&     0.421E-9 &  5.6&     0.786E-7 &  3.9&     0.127E-7 &  4.0 \\
  \hline  
\end{tabular} \end{center}  \end{table}

The solution in \eqref{u-1} is computed again by the weak Galerkin finite element
   $P_k$-$P_k$/$P_{k+2}$ (for $\{u_0, u_b\}$/$\nabla_w$), $k= 1,2,3,4$, on nonconvex polygonal  
    grids shown in Figure \ref{f22}.
The errors and the computed orders of convergence are listed in Tables \ref{t5}--\ref{t8}.
The optimal order of convergence is obtained for all cases and in all norms.
Most results are slightly worse than those on triangular grids.
But for some unknown reasons, the results for the $P_2$ element are much better.
  
  \begin{figure}[H]
 \begin{center}\setlength\unitlength{1.0pt}
\begin{picture}(360,120)(0,0)
  \put(15,108){$G_1$:} \put(125,108){$G_2$:} \put(235,108){$G_3$:} 
  \put(0,-420){\includegraphics[width=380pt]{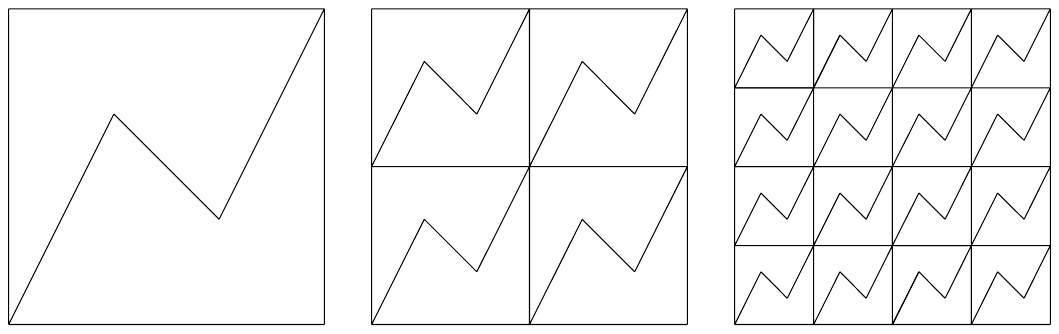}}  
 \end{picture}\end{center}
\caption{The nonconvex polygonal grids for the computation in Tables \ref{t5}--\ref{t8}. }\label{f22}
\end{figure}

  \begin{table}[H]
  \caption{By the $P_1$-$P_1$/$P_3$ element for \eqref{u-1}  on Figure \ref{f22} grids.} \label{t5}
\begin{center}  
   \begin{tabular}{c|rr|rr|rr}  
 \hline 
$G_i$ &  $ \|Q_h  u -   u_h \| $ & $O(h^r)$ &  $ \| \nabla_w( Q_h u- u_h )\| $ & $O(h^r)$ 
   &  $ \3bar  Q_h u- u_h  \3bar $ & $O(h^r)$  \\ \hline 
   &\multicolumn{6}{c}{$\lambda=1$ in \eqref{lambda} }\\  \hline 
 5&     0.702E-4 &  1.9&     0.715E-2 &  1.0&     0.190E-2 &  1.0 \\
 6&     0.185E-4 &  1.9&     0.357E-2 &  1.0&     0.949E-3 &  1.0 \\
 7&     0.489E-5 &  1.9&     0.178E-2 &  1.0&     0.474E-3 &  1.0 \\
 \hline  &\multicolumn{6}{c}{$\lambda=100$ in \eqref{lambda} }\\
 \hline 
 5&     0.386E-3 &  1.3&     0.831E-2 &  1.0&     0.194E-2 &  1.1 \\
 6&     0.143E-3 &  1.4&     0.397E-2 &  1.1&     0.956E-3 &  1.0 \\
 7&     0.439E-4 &  1.7&     0.188E-2 &  1.1&     0.475E-3 &  1.0 \\
  \hline  
\end{tabular} \end{center}  \end{table}
  
  \begin{table}[H]
  \caption{By the $P_2$-$P_2$/$P_4$ element for \eqref{u-1}  on Figure \ref{f22} grids.} \label{t6}
\begin{center}  
   \begin{tabular}{c|rr|rr|rr}  
 \hline 
$G_i$ &  $ \|Q_h  u -   u_h \| $ & $O(h^r)$ &  $ \| \nabla_w( Q_h u- u_h )\| $ & $O(h^r)$ 
   &  $ \3bar  Q_h u- u_h  \3bar $ & $O(h^r)$  \\ \hline 
   &\multicolumn{6}{c}{$\lambda=1$ in \eqref{lambda} }\\  \hline 
 3&     0.344E-4 &  2.9&     0.132E-2 &  2.1&     0.377E-3 &  2.9 \\
 4&     0.442E-5 &  3.0&     0.339E-3 &  2.0&     0.854E-4 &  2.1 \\
 5&     0.555E-6 &  3.0&     0.854E-4 &  2.0&     0.212E-4 &  2.0 \\
 \hline  &\multicolumn{6}{c}{$\lambda=100$ in \eqref{lambda} }\\
 \hline 
 3&     0.832E-4 &  4.4&     0.219E-2 &  4.2&     0.101E-2 &  4.9 \\
 4&     0.105E-4 &  3.0&     0.388E-3 &  2.5&     0.885E-4 &  3.5 \\
 5&     0.965E-6 &  3.4&     0.872E-4 &  2.2&     0.210E-4 &  2.1 \\
  \hline  
\end{tabular} \end{center}  \end{table}

  \begin{table}[H]
  \caption{By the $P_3$-$P_3$/$P_5$ element for \eqref{u-1}  on Figure \ref{f22} grids.} \label{t7}
\begin{center}  
   \begin{tabular}{c|rr|rr|rr}  
 \hline 
$G_i$ &  $ \|Q_h  u -   u_h \| $ & $O(h^r)$ &  $ \| \nabla_w( Q_h u- u_h )\| $ & $O(h^r)$ 
   &  $ \3bar  Q_h u- u_h  \3bar $ & $O(h^r)$  \\ \hline 
   &\multicolumn{6}{c}{$\lambda=1$ in \eqref{lambda} }\\  \hline 
 3&     0.178E-5 &  3.5&     0.753E-4 &  2.7&     0.114E-4 &  4.3 \\
 4&     0.134E-6 &  3.7&     0.111E-4 &  2.8&     0.113E-5 &  3.3 \\
 5&     0.922E-8 &  3.9&     0.151E-5 &  2.9&     0.134E-6 &  3.1 \\
 \hline  &\multicolumn{6}{c}{$\lambda=100$ in \eqref{lambda} }\\
 \hline 
 3&     0.134E-4 &  5.2&     0.161E-3 &  5.2&     0.660E-4 &  6.0 \\
 4&     0.245E-6 &  5.8&     0.109E-4 &  3.9&     0.156E-5 &  5.4 \\
 5&     0.907E-8 &  4.8&     0.145E-5 &  2.9&     0.138E-6 &  3.5 \\
  \hline  
\end{tabular} \end{center}  \end{table}

  \begin{table}[H]
  \caption{By the $P_4$-$P_4$/$P_6$ element for \eqref{u-1}  on Figure \ref{f22} grids.} \label{t8}
\begin{center}  
   \begin{tabular}{c|rr|rr|rr}  
 \hline 
$G_i$ &  $ \|Q_h  u -   u_h \| $ & $O(h^r)$ &  $ \| \nabla_w( Q_h u- u_h )\| $ & $O(h^r)$ 
   &  $ \3bar  Q_h u- u_h  \3bar $ & $O(h^r)$  \\ \hline 
   &\multicolumn{6}{c}{$\lambda=1$ in \eqref{lambda} }\\  \hline 
 2&     0.120E-5 &  7.2&     0.648E-4 &  5.7&     0.396E-4 &  5.4 \\
 3&     0.244E-7 &  5.6&     0.236E-5 &  4.8&     0.668E-6 &  5.9 \\
 4&     0.406E-8 &  ---&     0.163E-6 &  3.9&     0.165E-7 &  5.3 \\
 \hline  &\multicolumn{6}{c}{$\lambda=100$ in \eqref{lambda} }\\
 \hline 
 2&     0.156E-3 &  5.4&     0.186E-2 &  5.5&     0.673E-3 &  7.0 \\
 3&     0.659E-6 &  7.9&     0.146E-4 &  7.0&     0.543E-5 &  7.0 \\
 4&     0.897E-8 &  6.2&     0.219E-6 &  6.1&     0.460E-7 &  6.9 \\ 
  \hline  
\end{tabular} \end{center}  \end{table}
  
Next we solve the 3D convection equation \eqref{model} 
   on a cub domain $\Omega= (-1,1)\times (-1,1)\times (-1,1)$, where  
\a{ \bbeta =\p{1\\1\\1}, \quad c(x,y) =- 3\lambda, \quad f=0.  }
This is a more realistic problem, i.e., $f=0$, that the usual bad behavior in flow cannot be
  eliminated by the artificial source.
By choosing $g$ in \eqref{model}, the exact solution is 
\an{\label{u-2} u=\t{e}^{\lambda(x+y+z)}. } 
             
\begin{figure}[H]
 \begin{center}\setlength\unitlength{1.0pt}
\begin{picture}(360,120)(0,0)
  \put(15,108){$G_1$:} \put(125,108){$G_2$:} \put(235,108){$G_3$:} 
  \put(0,-420){\includegraphics[width=380pt]{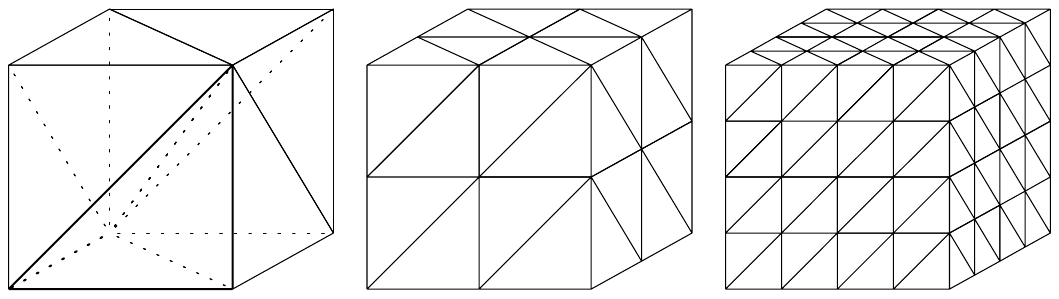}}  
 \end{picture}\end{center}
\caption{The tetrahedral grids for the computation in Tables \ref{t9}--\ref{t11}. }\label{f23}
\end{figure}
 
The solution in \eqref{u-2} is computed by the weak Galerkin finite element
   $P_k$-$P_k$/$P_{k+1}$ (for $\{u_0, u_b\}$/$\nabla_w$), $k= 1,2,3$, on tetrahedral
    grids shown in Figure \ref{f23}.
The errors and the computed orders of convergence are listed in Tables \ref{t9}--\ref{t11}.
The optimal order of convergence is obtained for all cases and in all norms. 
   
  \begin{table}[H]
  \caption{By the $P_1$-$P_1$/$P_2$ element for \eqref{u-2}  on Figure \ref{f23} grids.} \label{t9}
\begin{center}  
   \begin{tabular}{c|rr|rr|rr}  
 \hline 
$G_i$ &  $ \|Q_h  u -   u_h \| $ & $O(h^r)$ &  $ \| \nabla_w( Q_h u- u_h )\| $ & $O(h^r)$ 
   &  $ \3bar  Q_h u- u_h  \3bar $ & $O(h^r)$  \\ \hline 
   &\multicolumn{6}{c}{$\lambda=1$ in \eqref{u-2} }\\  \hline 
 3 &    0.916E-2 &2.37 &    0.594E+0 &1.02 &    0.576E+0 &0.98 \\
 4 &    0.199E-2 &2.20 &    0.296E+0 &1.01 &    0.287E+0 &1.00 \\
 5 &    0.469E-3 &2.08 &    0.148E+0 &1.00 &    0.143E+0 &1.01 \\
 \hline  &\multicolumn{6}{c}{$\lambda=2$ in \eqref{u-2} }\\
 \hline 
 3 &    0.396E+0 &2.54 &    0.209E+2 &1.08 &    0.199E+2 &0.99 \\
 4 &    0.768E-1 &2.37 &    0.103E+2 &1.03 &    0.996E+1 &1.00 \\
 5 &    0.171E-1 &2.17 &    0.509E+1 &1.01 &    0.495E+1 &1.01 \\
  \hline  
\end{tabular} \end{center}  \end{table}

  \begin{table}[H]
  \caption{By the $P_2$-$P_2$/$P_3$ element for \eqref{u-2}  on Figure \ref{f23} grids.} \label{t10}
\begin{center}  
   \begin{tabular}{c|rr|rr|rr}  
 \hline 
$G_i$ &  $ \|Q_h  u -   u_h \| $ & $O(h^r)$ &  $ \| \nabla_w( Q_h u- u_h )\| $ & $O(h^r)$ 
   &  $ \3bar  Q_h u- u_h  \3bar $ & $O(h^r)$  \\ \hline 
   &\multicolumn{6}{c}{$\lambda=1$ in \eqref{u-2} }\\  \hline 
 2 &    0.171E-2 &4.39 &    0.298E-1 &3.28 &    0.215E-1 &2.91 \\
 3 &    0.102E-3 &4.07 &    0.345E-2 &3.11 &    0.273E-2 &2.98 \\
 4 &    0.636E-5 &4.00 &    0.417E-3 &3.05 &    0.343E-3 &2.99 \\
 \hline  &\multicolumn{6}{c}{$\lambda=2$ in \eqref{u-2} }\\
 \hline 
 2 &    0.381E+0 &2.81 &    0.473E+1 &2.41 &    0.277E+1 &2.64 \\
 3 &    0.102E+0 &1.90 &    0.152E+1 &1.64 &    0.370E+0 &2.90 \\
 4 &    0.113E-1 &3.17 &    0.279E+0 &2.44 &    0.470E-1 &2.98 \\
  \hline  
\end{tabular} \end{center}  \end{table}

  \begin{table}[H]
  \caption{By the $P_3$-$P_3$/$P_4$ element for \eqref{u-2}  on Figure \ref{f23} grids.} \label{t11}
\begin{center}  
   \begin{tabular}{c|rr|rr|rr}  
 \hline 
$G_i$ &  $ \|Q_h  u -   u_h \| $ & $O(h^r)$ &  $ \| \nabla_w( Q_h u- u_h )\| $ & $O(h^r)$ 
   &  $ \3bar  Q_h u- u_h  \3bar $ & $O(h^r)$  \\ \hline 
   &\multicolumn{6}{c}{$\lambda=1$ in \eqref{u-2} }\\  \hline 
 2 &    0.171E-2 &4.39 &    0.298E-1 &3.28 &    0.215E-1 &2.91 \\
 3 &    0.102E-3 &4.07 &    0.345E-2 &3.11 &    0.273E-2 &2.98 \\
 4 &    0.636E-5 &4.00 &    0.417E-3 &3.05 &    0.343E-3 &2.99 \\
 \hline  &\multicolumn{6}{c}{$\lambda=2$ in \eqref{u-2} }\\
 \hline 
 2 &    0.381E+0 &2.81 &    0.473E+1 &2.41 &    0.277E+1 &2.64 \\
 3 &    0.102E+0 &1.90 &    0.152E+1 &1.64 &    0.370E+0 &2.90 \\
 4 &    0.113E-1 &3.17 &    0.279E+0 &2.44 &    0.470E-1 &2.98 \\
  \hline  
\end{tabular} \end{center}  \end{table}

\begin{figure}[H]
 \begin{center}\setlength\unitlength{1.0pt}
\begin{picture}(360,120)(0,0)
  \put(15,108){$G_1$:} \put(125,108){$G_2$:} \put(235,108){$G_3$:} 
  \put(0,-420){\includegraphics[width=380pt]{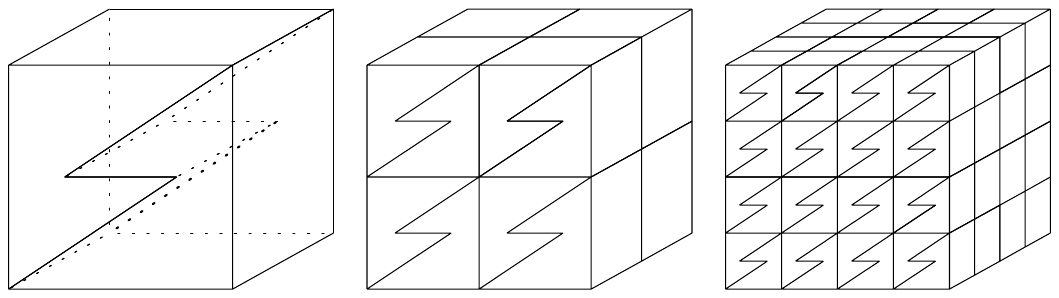}}  
 \end{picture}\end{center}
\caption{The nonconvex polyhedral grids for the computation in Tables \ref{t12}--\ref{t14}. }\label{f24}
\end{figure}
 
The solution in \eqref{u-2} is computed again by the weak Galerkin finite element
   $P_k$-$P_k$/$P_{k+2}$ (for $\{u_0, u_b\}$/$\nabla_w$), $k= 1,2,3$, on nonconvex polyhedral 
    grids shown in Figure \ref{f24}.
The errors and the computed orders of convergence are listed in Tables \ref{t12}--\ref{t14}.
The optimal order of convergence is obtained for all cases and in all norms. 
The results on nonconvex polyhedral grids are much better, as we have much less unknowns.
   
  \begin{table}[H]
  \caption{By the $P_1$-$P_1$/$P_3$ element for \eqref{u-2}  on Figure \ref{f24} grids.} \label{t12}
\begin{center}  
   \begin{tabular}{c|rr|rr|rr}  
 \hline 
$G_i$ &  $ \|Q_h  u -   u_h \| $ & $O(h^r)$ &  $ \| \nabla_w( Q_h u- u_h )\| $ & $O(h^r)$ 
   &  $ \3bar  Q_h u- u_h  \3bar $ & $O(h^r)$  \\ \hline 
   &\multicolumn{6}{c}{$\lambda=1$ in \eqref{u-2} }\\  \hline 
 4 &    0.256E-2 &2.39 &    0.286E+0 &1.08 &    0.303E+0 &1.13 \\
 5 &    0.597E-3 &2.10 &    0.141E+0 &1.03 &    0.147E+0 &1.05 \\
 6 &    0.149E-3 &2.00 &    0.699E-1 &1.01 &    0.725E-1 &1.02 \\
 \hline  &\multicolumn{6}{c}{$\lambda=2$ in \eqref{u-2} }\\
 \hline 
 4 &    0.294E+0 &2.74 &    0.106E+2 &1.24 &    0.115E+2 &1.39 \\
 5 &    0.437E-1 &2.75 &    0.494E+1 &1.10 &    0.521E+1 &1.15 \\
 6 &    0.734E-2 &2.57 &    0.242E+1 &1.03 &    0.252E+1 &1.05 \\
  \hline  
\end{tabular} \end{center}  \end{table}

  \begin{table}[H]
  \caption{By the $P_2$-$P_2$/$P_4$ element for \eqref{u-2}  on Figure \ref{f24} grids.} \label{t13}
\begin{center}  
   \begin{tabular}{c|rr|rr|rr}  
 \hline 
$G_i$ &  $ \|Q_h  u -   u_h \| $ & $O(h^r)$ &  $ \| \nabla_w( Q_h u- u_h )\| $ & $O(h^r)$ 
   &  $ \3bar  Q_h u- u_h  \3bar $ & $O(h^r)$  \\ \hline 
   &\multicolumn{6}{c}{$\lambda=1$ in \eqref{u-2} }\\  \hline 
 2 &    0.692E-2 &3.21 &    0.123E+0 &2.45 &    0.146E+0 &2.55 \\
 3 &    0.481E-3 &3.85 &    0.170E-1 &2.85 &    0.193E-1 &2.92 \\
 4 &    0.316E-4 &3.93 &    0.219E-2 &2.96 &    0.238E-2 &3.02 \\
 \hline  &\multicolumn{6}{c}{$\lambda=2$ in \eqref{u-2} }\\
 \hline 
 2 &    0.347E+1 &1.66 &    0.174E+2 &1.06 &    0.147E+2 &1.69 \\
 3 &    0.467E+0 &2.90 &    0.332E+1 &2.39 &    0.246E+1 &2.58 \\
 4 &    0.407E-1 &3.52 &    0.433E+0 &2.94 &    0.327E+0 &2.91 \\
  \hline  
\end{tabular} \end{center}  \end{table}

  \begin{table}[H]
  \caption{By the $P_3$-$P_3$/$P_5$ element for \eqref{u-2}  on Figure \ref{f24} grids.} \label{t14}
\begin{center}  
   \begin{tabular}{c|rr|rr|rr}  
 \hline 
$G_i$ &  $ \|Q_h  u -   u_h \| $ & $O(h^r)$ &  $ \| \nabla_w( Q_h u- u_h )\| $ & $O(h^r)$ 
   &  $ \3bar  Q_h u- u_h  \3bar $ & $O(h^r)$  \\ \hline 
   &\multicolumn{6}{c}{$\lambda=1$ in \eqref{u-2} }\\  \hline 
 2 &    0.137E-2 &4.88 &    0.260E-1 &3.35 &    0.294E-1 &3.33 \\
 3 &    0.494E-4 &4.79 &    0.208E-2 &3.64 &    0.222E-2 &3.73 \\
 4 &    0.167E-5 &4.89 &    0.148E-3 &3.81 &    0.150E-3 &3.88 \\
 \hline  &\multicolumn{6}{c}{$\lambda=2$ in \eqref{u-2} }\\
 \hline  
 2 &    0.120E+0 &3.58 &    0.166E+1 &2.48 &    0.700E+0 &3.51 \\
 3 &    0.865E-2 &3.79 &    0.255E+0 &2.70 &    0.483E-1 &3.86 \\
 4 &    0.252E-3 &5.10 &    0.201E-1 &3.66 &    0.310E-2 &3.96 \\
  \hline  
\end{tabular} \end{center}  \end{table}


\begin{thebibliography}{99}
 

  

 
 

 \bibitem{wy3655}
{\sc J. Wang and X. Ye}, {\em A weak Galerkin mixed finite element
method for second-order elliptic problems},  Math. Comp., 83 (2014), pp. 2101-2126.
 
 

 
   

 
   


 

 
 

 

   

  
\bibitem{a2001} {\sc R. Abgrall}, {\em Toward the ultimate conservative scheme: following the quest}, J. Comput. Phys., 167 (2001), pp. 277-315.
 

\bibitem{b1999} {\sc T. Barth}, {\em Numerical methods for gasdynamic systems on unstructured meshes, in An Introduction to Recent Developments in Theory and Numerics for Conservation Laws} (Freiburg/Littenweiler, 1997), Lect. Notes Comput. Sci. Eng. 5, Springer-Verlag, Berlin, 1999, pp. 195-285.

\bibitem{bd1999} {\sc T. Barth, and H. Deconinck, eds.}, {\em  High-order methods for computational physics}, Lect. Notes Comput. Sci. Eng. 9, Springer-Verlag, Berlin, 1999.

\bibitem{bc2001} {\sc P. Bochev, and  J. Choi}, {\em  Improved least-squares error estimates for scalar hyperbolic problems}, Comput. Methods Appl. Math., 1 (2001), pp. 115-124.

\bibitem{burman2014} {\sc E. Burman},
{\em Stabilized finite element methods for nonsymmetric,
noncoercive, and ill-possed problems. Part II: hyperbolic
equations}, SIAM J. Sci. Comput, vol. 36, No. 4, pp.
A1911-A1936, 2014.
 
 \bibitem{b} {\sc E. Burman}, {\em A unified analysis for conforming and nonconforming stabilized finite element methods using interior penalty}, SIAM J. Numer. Anal. 43(5), 2012-2033 (2005).
  

\bibitem{be} {\sc E. Burman, and A. Ern}, {\em Continuous interior penalty hp-finite element methods for advection and advection diffusion equations}, Math. Comp., 76, 1119-1140 (2007).

 \bibitem{bh} {\sc E. Burman, and P. Hansbo}, {\em Edge stabilization for Galerkin approximations of convection-diffusionreaction problems}, Comput. Methods Appl. Mech. Eng. 193(15–16), 1437-1453 (2004).

\bibitem{bs} {\sc  E. Burman, and B. Stamm}, {\em Discontinuous and continuous finite element methods with interior penalty for hyperbolic problems}, Tech. Report, EPFL-IACS report 17 (2005).

\bibitem{bs2007} {\sc  E. Burman, and B. Stamm}, {\em Minimal Stabilization for Discontinuous Galerkin Finite Element Methods for Hyperbolic Problems}, J Sci Comput (2007) 33: 183-208.

\bibitem{cj1988} {\sc G. Carey, and B. Jiang}, {\em  Least-squares finite elements for first-order hyperbolic systems}, Internat. J. Numer. Methods Engrg., 26 (1988), pp. 81-93.

\bibitem{c1999} {\sc B. Cockburn}, {\em Discontinuous Galerkin methods for convection-dominated problems}. In: High-Order Methods for Computational Physics. Lect. Notes Comput. Sci. Eng., vol. 9, pp. 69-224. Springer, Berlin
(1999).
 

\bibitem{cs} {\sc  B. Cockburn, and C. Shu}, {\em Runge-Kutta discontinuous Galerkin methods for convection-dominated problems}, J. Sci. Comput. 16(3), 173-261 (2001).

 

 \bibitem{pdwg3} {\sc W. Cao, C. Wang and J. Wang},  {\em An $L^p$-Primal-Dual Weak Galerkin Method for div-curl Systems}, Journal of Computational and Applied Mathematics, vol. 422, 114881, 2023.
 \bibitem{pdwg4}{\sc  W. Cao, C. Wang and J. Wang},  {\em An $L^p$-Primal-Dual Weak Galerkin Method for Convection-Diffusion Equations}, Journal of Computational and Applied Mathematics, vol. 419, 114698, 2023. 
 \bibitem{pdwg5}{\sc W. Cao, C. Wang and J. Wang},  {\em A New Primal-Dual Weak Galerkin Method for Elliptic Interface Problems with Low Regularity Assumptions}, Journal of Computational Physics, vol. 470, 111538, 2022.
   \bibitem{wg11}{\sc  S. Cao, C. Wang and J. Wang},  {\em A new numerical method for div-curl Systems with Low Regularity Assumptions}, Computers and Mathematics with Applications, vol. 144, pp. 47-59, 2022.
\bibitem{pdwg10}{\sc  W. Cao and C. Wang},  {\em New Primal-Dual Weak Galerkin Finite Element Methods for Convection-Diffusion Problems}, Applied Numerical Mathematics, vol. 162, pp. 171-191, 2021. 

  


\bibitem{eehj1996} {\sc K. Eriksson, D. Estep, P. Hansbo, and C. Johnson}, {\em  Computational Differential Equations}, Cambridge University Press, Cambridge, UK, 1996.

\bibitem{eg} {\sc  A. Ern, and J. Guermond}, {\em Discontinuous Galerkin methods for Friedrichs’ systems. i. General theory}, SIAM J. Numer. Anal. 44(2), 753-778 (2006).

 

\bibitem{hjs2002} {\sc P. Houston, M. Jensen, and E. Suli}, {\em  Hp-discontinuous Galerkin finite element methods with least-squares stabilization}, J. Sci. Comput., 17 (2002), pp. 3-25.

\bibitem{hss} {\sc  P. Houston,  C. Schwab, and E. Suli}, {\em Discontinuous hp-finite element methods for advection-diffusionreaction problems}, SIAM J. Numer. Anal. 39(6), 2133-2163 (2002).

\bibitem{j2004} {\sc M. Jensen}, {\em Discontinuous Galerkin methods for Friedrichs’ systems with irregular solutions}. Ph.D. thesis, University of Oxford (2004).

\bibitem{jnp1984} {\sc C. Johnson, U. Navert, and J. Pitkaranta}, {\em Finite element methods for linear hyperbolic problems}, Comput. Methods Appl. Mech. Engrg., 45 (1984), pp. 285-312.
 
\bibitem{lr} {\sc P. Lesaint, and P. Raviart}, {\em On a finite element method for solving the neutron transport equation}, Mathematical Aspects of Finite Elements in Partial Differential Equations, Proc. Sympos., Math. Res. Center, Univ. Wisconsin, Madison,Wis., 1974. Math. Res. Center, Univ. ofWisconsin-Madison, vol. 33, pp. 89-123. Academic, New York (1974).

\bibitem{l1992} {\sc R. LeVeque}, {\em Numerical Methods for Conservation Laws}, 2nd ed., Lectures in Mathematics ETH Zurich, Birkhauser Verlag, Basel, 1992.

\bibitem{lm1993} {\sc E. Lewis, and J. Miller}, {\em Computational Methods of Neutron Transport}, American Nuclear Society, La Grange Park, IL, 1993.

 
 \bibitem{wg14}{\sc D. Li, Y. Nie, and C. Wang},  {\em Superconvergence of Numerical Gradient for Weak Galerkin Finite Element Methods on Nonuniform Cartesian Partitions in Three Dimensions}, Computers and Mathematics with Applications, vol 78(3), pp. 905-928, 2019.  
  \bibitem{wg1} {\sc D. Li, C. Wang and J. Wang},  {\em An Extension of the Morley Element on General Polytopal Partitions Using Weak Galerkin Methods}, Journal of Scientific Computing, 100, vol 27, 2024.  
 \bibitem{wg2} {\sc D. Li, C. Wang and S. Zhang},  {\em Weak Galerkin methods for elliptic interface problems on curved polygonal partitions}, Journal of Computational and Applied Mathematics, pp. 115995, 2024. 
\bibitem{wg5} {\sc D. Li, C. Wang, J.  Wang and X. Ye},  {\em Generalized weak Galerkin finite element methods for second order elliptic problems}, Journal of Computational and Applied Mathematics, vol. 445, pp. 115833, 2024.
 \bibitem{wg6} {\sc D. Li, C. Wang, J. Wang and S. Zhang},  {\em High Order Morley Elements for Biharmonic Equations on Polytopal Partitions}, Journal of Computational and Applied Mathematics, Vol. 443, pp. 115757, 2024.
 \bibitem{wg7} {\sc D. Li, C. Wang and J. Wang},  {\em Curved Elements in Weak Galerkin Finite Element Methods}, Computers and Mathematics with Applications, Vol. 153, pp. 20-32, 2024.
\bibitem{wg8} {\sc D. Li, C. Wang and J. Wang},  {\em Generalized Weak Galerkin Finite Element Methods for Biharmonic Equations}, Journal of Computational and Applied Mathematics, vol. 434, 115353, 2023.
 \bibitem{pdwg1} {\sc D. Li, C. Wang and J. Wang},  {\em An $L^p$-primal-dual finite element method for first-order transport problems}, Journal of Computational and Applied Mathematics, vol. 434, 115345, 2023.
 \bibitem{pdwg2} {\sc D. Li and C. Wang},  {\em A simplified primal-dual weak Galerkin finite element method for Fokker-Planck type equations}, Journal of Numerical Methods for Partial Differential Equations, vol 39, pp. 3942-3963, 2023.
\bibitem{pdwg6}{\sc  D. Li, C. Wang and J. Wang},  {\em Primal-Dual Weak Galerkin Finite Element Methods for Transport Equations in Non-Divergence Form}, Journal of Computational and Applied Mathematics, vol. 412, 114313, 2022.
  \bibitem{wg13}{\sc  D. Li, C. Wang, and J. Wang},  {\em Superconvergence of the Gradient Approximation for Weak Galerkin Finite Element Methods on Rectangular Partitions}, Applied Numerical Mathematics, vol. 150, pp. 396-417, 2020.
 

\bibitem{rh} {\sc W. Reed, and T. Hill}, {\em Triangular mesh methods for the neutron transport equation}, Tech. Report LAUR- 73-479, Los Alamos Scientific Laboratory (1973).

\bibitem{smmo2004} {\sc H. Sterck, T. Manteuffel, S. Mccormick, and L. Olson}, {\em Least-squares finite element methods and algebraic multigrid solvers for linear hyperbolic PDEs}, SIAM J. SCI. COMPUT.,  Vol. 26, No. 1, pp. 31-54.

 
 \bibitem{wg15}{\sc C. Wang},  {\em New Discretization Schemes for Time-Harmonic Maxwell Equations by Weak Galerkin Finite Element Methods}, Journal of Computational and Applied Mathematics, Vol. 341, pp. 127-143, 2018.  
 
   \bibitem{pdwg7}{\sc  C. Wang},  {\em Low Regularity Primal-Dual Weak Galerkin Finite Element Methods for Ill-Posed Elliptic Cauchy Problems}, Int. J. Numer. Anal. Mod., vol. 19(1), pp. 33-51, 2022.
 
 \bibitem{pdwg8}{\sc  C. Wang},  {\em A Modified Primal-Dual Weak Galerkin Finite Element Method for Second Order Elliptic Equations in Non-Divergence Form}, Int. J. Numer. Anal. Mod., vol. 18(4), pp. 500-523, 2021.

 
 \bibitem{pdwg13}{\sc  C. Wang},  {\em A New Primal-Dual Weak Galerkin Finite Element Method for Ill-posed Elliptic Cauchy Problems}, Journal of Computational and Applied Mathematics, vol 371, 112629, 2020.
 
  
  
   
 \bibitem{pdwg11}{\sc  C. Wang and J. Wang},  {\em A Primal-Dual Weak Galerkin Finite Element Method for Fokker-Planck Type Equations}, SIAM Numerical Analysis, vol. 58(5), pp. 2632-2661, 2020.
 \bibitem{pdwg12}{\sc  C. Wang and J. Wang},  {\em A Primal-Dual Finite Element Method for First-Order Transport Problems}, Journal of Computational Physics, Vol. 417, 109571, 2020.
  
 \bibitem{pdwg14}{\sc  C. Wang and J. Wang},  {\em Primal-Dual Weak Galerkin Finite Element Methods for Elliptic Cauchy Problems}, Computers and Mathematics with Applications, vol 79(3), pp. 746-763, 2020. 
 \bibitem{pdwg15}{\sc  C. Wang and J. Wang},  {\em A Primal-Dual Weak Galerkin Finite Element Method for Second Order Elliptic Equations in Non-Divergence form}, Mathematics of Computation, Vol. 87, pp. 515-545, 2018.  
   
\bibitem{wwhyperbolic} {\sc C. Wang, and J. Wang}, {\em A PRIMAL-DUAL FINITE ELEMENT METHOD FOR FIRST-ORDER TRANSPORT PROBLEMS}, arxiv. 1906.07336.


 \bibitem{wg17}{\sc C. Wang and J. Wang},  {\em Discretization of Div-Curl Systems by Weak Galerkin Finite Element Methods on Polyhedral Partitions}, Journal of Scientific Computing, Vol. 68, pp. 1144-1171, 2016.    
   \bibitem{wg19}{\sc C. Wang and J. Wang},  {\em A Hybridized Formulation for Weak Galerkin Finite Element Methods for Biharmonic Equation on Polygonal or Polyhedral Meshes}, International Journal of Numerical Analysis and Modeling, Vol. 12, pp. 302-317, 2015. 
 \bibitem{wg20}{\sc  J. Wang and C. Wang},  {\em Weak Galerkin Finite Element Methods for Elliptic PDEs}, Science China, Vol. 45, pp. 1061-1092, 2015.  
 \bibitem{wg21}{\sc C. Wang and J. Wang},  {\em An Efficient Numerical Scheme for the Biharmonic Equation by Weak Galerkin Finite Element Methods on Polygonal or Polyhedral Meshes}, Journal of Computers and Mathematics with Applications, Vol. 68, 12, pp. 2314-2330, 2014.  
 
   \bibitem{wg18}{\sc C. Wang, J. Wang, R. Wang and R. Zhang},  {\em A Locking-Free Weak Galerkin Finite Element Method for Elasticity Problems in the Primal Formulation}, Journal of Computational and Applied Mathematics, Vol. 307, pp. 346-366, 2016.   
 

 
 \bibitem{wg12}{\sc  C. Wang, J. Wang, X. Ye and S. Zhang},  {\em De Rham Complexes for Weak Galerkin Finite Element Spaces}, Journal of Computational and Applied Mathematics, vol. 397, pp. 113645, 2021.
 
 \bibitem{wg3} {\sc C. Wang, J. Wang and S. Zhang},  {\em Weak Galerkin Finite Element Methods for Optimal Control Problems Governed by Second Order Elliptic Partial Differential Equations}, Journal of Computational and Applied Mathematics, in press, 2024. 
 
 \bibitem{itera} {\sc C. Wang, J. Wang and S. Zhang},  {\em A parallel iterative procedure for weak Galerkin methods for second order elliptic problems}, International Journal of Numerical Analysis and Modeling, vol. 21(1), pp. 1-19, 2023.
 \bibitem{wg9} {\sc C. Wang, J. Wang and S. Zhang},  {\em Weak Galerkin Finite Element Methods for Quad-Curl Problems}, Journal of Computational and Applied Mathematics, vol. 428, pp. 115186, 2023.

  

\bibitem{ellip_JCAM2013} {\sc J. Wang and X. Ye}, {\em
A weak Galerkin finite element method for second-order elliptic problems}, J. Comput. Appl. Math., vol. 241, pp. 103-115, 2013.
 



\bibitem{wg4} {\sc C. Wang, X. Ye and S. Zhang},  {\em A Modified weak Galerkin finite element method for the Maxwell equations on polyhedral meshes}, Journal of Computational and Applied Mathematics, vol. 448, pp. 115918, 2024. 
 
 \bibitem{wz2023} {\sc C. Wang and S. Zhang}, {\em
A Weak Galerkin Method for Elasticity Interface Problems}, Journal of Computational and Applied Mathematics, vol. 419, 114726, 2023. 

   


   \bibitem{wg10}{\sc  C. Wang and S. Zhang},  {\em A Weak Galerkin Method for Elasticity Interface Problems}, Journal of Computational and Applied Mathematics, vol. 419, 114726, 2023. 
  \bibitem{pdwg9}{\sc  C. Wang and L. Zikatanov},  {\em Low Regularity Primal-Dual Weak Galerkin Finite Element Methods for Convection-Diffusion Equations}, Journal of Computational and Applied Mathematics, vol 394, 113543, 2021.
 
 \bibitem{wg16}{\sc  C. Wang and H. Zhou},  {\em A Weak Galerkin Finite Element Method for a Type of Fourth Order Problem arising from Fluorescence Tomography}, Journal of Scientific Computing, Vol. 71(3), pp. 897-918, 2017.  
  
\end{thebibliography}
\end{document}